\UseRawInputEncoding
\documentclass[12pt]{article}

\usepackage{dsfont}
\usepackage{amsfonts,amsmath,amsthm}
\usepackage{algorithm, algorithmic}
\usepackage{color}
\usepackage{graphicx}
\usepackage{stmaryrd}        

\usepackage{multicol}
\usepackage{multirow}

\usepackage[paper=a4paper,dvips,top=2cm,left=2cm,right=2cm,
foot=1cm,bottom=4cm]{geometry}

\newcommand{\ve}{{\bf e}}

\begin{document}
	\large
	
	\title{A Genuine Extension of The Moore-Penrose Inverse to Dual Matrices}
	\author{Chunfeng Cui\footnote{LMIB of the Ministry of Education, School of Mathematical Sciences, Beihang University, Beijing 100191 China.
			({\tt chungfengcui@buaa.edu.cn}).}
		\and \
		Liqun Qi\footnote{Department of Applied Mathematics, The Hong Kong Polytechnic University, Hung Hom, Kowloon, Hong Kong; Department of Mathematics, School of Science, Hangzhou Dianzi University, Hangzhou 310018 China
			({\tt maqilq@polyu.edu.hk}).}
	}
	\date{\today}
	\maketitle

	\begin{abstract}
		The Moore-Penrose inverse is a genuine extension of the matrix inverse.  Given a complex matrix, there uniquely exists another complex matrix satisfying the four Moore-Penrose conditions, and if the original matrix is nonsingular, it is exactly the inverse of that matrix.  In the last one and a half decades, in the study of approximate synthesis in kinematic, two generalizations of the Moore-Penrose inverse appeared for dual real matrices,
		including Moore-Penrose dual generalized inverse and
		dual Moore-Penrose generalized inverse (DMPGI).
		DMPGI satisfies the four Moore-Penrose conditions, but does not exist for  uncountably many  dual real matrices.  
		In this paper,  based on the singular value decomposition of dual matrices,
		we  extend the  first  Moore-Penrose condition  to dual matrices and  introduce a genuine extension of the Moore-Penrose inverse to dual matrices, referred to as GMPI.
		Given a dual complex matrix, its GMPI is the unique dual complex matrix that satisfies the  first extended and the other three  Moore-Penrose conditions.  If the original matrix is a complex matrix, its GMPI  is exactly the Moore-Penrose inverse of that matrix. And if the original matrix is a dual real matrix and its DMPGI exists,  its GMPI  coincides with its  DMPGI.

		\medskip


		
	\end{abstract}

	\renewcommand{\Re}{\mathds{R}}
	\newcommand{\rank}{\mathrm{rank}}
	\renewcommand{\span}{\mathrm{span}}
	\newcommand{\X}{\mathcal{X}}
	\newcommand{\A}{\mathcal{A}}
	\newcommand{\I}{\mathcal{I}}
	\newcommand{\B}{\mathcal{B}}
	\newcommand{\C}{\mathcal{C}}
	\newcommand{\OO}{\mathcal{O}}
	\newcommand{\e}{\mathbf{e}}
	\newcommand{\0}{\mathbf{0}}
	\newcommand{\dd}{\mathbf{d}}
	\newcommand{\ii}{\mathbf{i}}
	\newcommand{\jj}{\mathbf{j}}
	\newcommand{\kk}{\mathbf{k}}
	\newcommand{\va}{\mathbf{a}}
	\newcommand{\vb}{\mathbf{b}}
	\newcommand{\vc}{\mathbf{c}}
	\newcommand{\vq}{\mathbf{q}}
	\newcommand{\vg}{\mathbf{g}}
	\newcommand{\pr}{\vec{r}}
	\newcommand{\ps}{\vec{s}}
	\newcommand{\pt}{\vec{t}}
	\newcommand{\pu}{\vec{u}}
	\newcommand{\pv}{\vec{v}}
	\newcommand{\pw}{\vec{w}}
	\newcommand{\pp}{\vec{p}}
	\newcommand{\pq}{\vec{q}}
	\newcommand{\pl}{\vec{l}}
	\newcommand{\vt}{\rm{vec}}
	\newcommand{\vx}{\mathbf{x}}
	\newcommand{\vy}{\mathbf{y}}
	\newcommand{\vu}{\mathbf{u}}
	\newcommand{\vv}{\mathbf{v}}
	\newcommand{\y}{\mathbf{y}}
	\newcommand{\vz}{\mathbf{z}}
	\newcommand{\T}{\top}
	
	\newtheorem{Thm}{Theorem}[section]
	\newtheorem{Def}{Definition}[section]
	\newtheorem{Ass}[Thm]{Assumption}
	\newtheorem{Lem}[Thm]{Lemma}
	\newtheorem{Prop}[Thm]{Proposition}
	\newtheorem{Cor}[Thm]{Corollary}
	\newtheorem{example}{Example}[section]
	\newtheorem{remark}[Thm]{Remark}
	
	\section{Introduction}
	
	The Moore-Penrose inverse is a genuine extension of the matrix inverse  \cite{BG03, WWQ18}.  It was independently discovered by
	E.H. Moore \cite{Mo20} in 1920, Arne Bjerhammar \cite{Bj51} in 1951, and Roger Penrose \cite{Pe55} in 1955.
	In \cite{Pe55}, young Penrose wrote:  This paper describes a generalization of the inverse of a non-singular matrix, as the unique solution of a certain set of equations. This generalized inverse exists for any
	(possibly rectangular) matrix whatsoever with complex elements.  Then he gave the four conditions as follows (Theorem 1 of \cite{Pe55}).    The four equations
	\begin{equation} \label{pen1}
		AXA = A,
	\end{equation}
	\begin{equation} \label{pen2}
		XAX = X,
	\end{equation}
	\begin{equation} \label{pen3}
		(AX)^* = AX,
	\end{equation}
	\begin{equation} \label{pen4}
		(XA)^* = XA,
	\end{equation}
	have a unique solution for any (complex matrix) $A$.   Such a solution is denoted as $A^+$, and called the Moore-Penrose inverse now.  The four conditions (\ref{pen1}-\ref{pen4}) are called the Moore Penrose conditions.   Roger Penrose obtained Nobel prize in physics in 2020.   This increases the mysteriousity of these four conditions.
	
	On the other hand, dual algebra found quite applications in kinematic analysis \cite{An98, PS07, PV09}.
	{An $m \times n$} dual real matrix $A$ can be written as $A = A_s + A_d\epsilon$, where $A_s$ and $A_d$ are $m \times n$ real matrices, $\epsilon$ is the infinitesimal unit, which satisfies that $\epsilon \not = 0$ and $\epsilon^2 = 0$.   In 2007, Pennestr\'{i} and Stefanelli \cite{PS07} introduced the Moore-Penrose inverse (MPDGI) $A^P$ for
	a dual real matrix $A$.  The MPDGI $A^P$ is then defined as
	\begin{equation} \label{MPDGI}
		A^P = A_s^+ - A_s^+A_dA_s^+\epsilon.
	\end{equation}
	Pennestr\'{i} and Valentini \cite{PV09} in 2009, Angeles \cite{An12} in 2012 further discussed the properties and applications of MPDGI.  MPDGI exists for any dual real matrix $A$, but its definition somewhat deviates from the Moore-Penrose conditions (\ref{pen1}-\ref{pen4}).  In 2018, de Falco, Pennestr\'{i} and Udwadia \cite{FPU18} studied generalized inverses of dual real matrices, which satisfy the four Moore-Penrose conditions (\ref{pen1}-\ref{pen4}).   In 2020, Udwadia, Pennestr\'{i} and de Falco \cite{UPF20} called a solution $X$ of (\ref{pen1}-\ref{pen4}) for a dual real matrix $A$, the dual Moore-Penrose generalized inverse (DMPGI) of $A$, and pointed out that there are uncountably many dual real matrices that do not have DMPGIs.  In 2021, Udwadia \cite{Ud21} analyzed conditions for a dual real matrix has generalized inverses which satisfy some of (\ref{pen1}-\ref{pen4}).  Then, Udwadia \cite{Ud21a} studied dual generalized inverses and their use in solving systems of linear dual equations.  Also in 2021, Wang \cite{Wa21} gave necessary and sufficient conditions for a dual real matrix $A$ having DMPGI, and conditions that the DMPGI and MPDGI of $A$ are equal.   The generalization of the Moore-Penrose inverse to dual real matrices attracted quite attention but seems there still needs a genuine extension of the Moore-Penrose inverse from complex matrices to dual matrices.   We now denote the DMPGI of a dual real matrix $A$ as $A^D$, not using the notation $A^+$ to avoid prejudice but unmatured conclusions.
	
	We may think about the issue of generalizing the Moore-Penrose inverse to dual matrices from a different angle.   Instead of requiring all dual matrices to satisfy the four conditions (\ref{pen1}-\ref{pen4}), we may think that the four conditions (\ref{pen1}-\ref{pen4}) maybe need to be modified for dual matrices.
	
	As pointed out by Udwadia, Pennestr\'{i} and de Falco \cite{UPF20}, the Moore-Penrose condition which yields difficulty to dual matrices is (\ref{pen1}).  Let $m=n=1$.  Then $A$ and $X$ are dual numbers.  If $A = a_d\epsilon$ is a nonzero infinitesimal dual number, where $a_d$ is a real or complex number, then we always have $AXA = 0$, i.e., (\ref{pen1}) cannot hold.  Thus, we may think to generalize (\ref{pen1}) to dual numbers.   Then, how to do this?
	
	As recently pointed out by Wang, Cui and Wei in \cite{WCW23}, dual numbers and dual generalized inverses are gradually discussed and studied by scholars, but the basic concepts of dual number algebra still lack.  It thus needs to explore basic knowledge of dual numbers and dual quaternions, which provides a solid foundation for subsequent research on dual algebra.
	
	The Moore-generalized inverse is closely related to the singular value decomposition {(SVD)} of matrices.  Recently, Qi and Luo \cite{QL23} presented the {SVD} of dual quaternion matrices, and {SVD}s of dual real matrices and dual complex matrices can be derived from there.   It was revealed that the singular values of dual matrices are dual real numbers.  The problem is that some of such
	singular values may be infinitesimal dual numbers, which have no reciprocals.   It is well known that the nonzero singular values of the Moore-Penrose inverse of a complex matrix $A$ are the reciprocals of nonzero singular values of $A$.  This cannot hold for dual numbers.   Hence, we consider an essential approximation $A_e$ of $A$, such that the nonzero singular values of $A_e$ are the appreciable singular values of $A$, and replace (\ref{pen1}) by
	\begin{equation} \label{pene}
		AXA = A_e.
	\end{equation}
	When $A$ are complex matrices, (\ref{pene}) reduces to (\ref{pen1}).   With this extension, we found that there is a unique dual matrix $X$ satisfying (\ref{pen2}-\ref{pen4}) and (\ref{pene}).  We call this solution as the genuine Moore-Penrose inverse (GMPI) of $A$, and denoted it as $A^G$.  We found that $A^G = A^D$ as long as $A^D$ exists, {and gave a new sufficient and necessary condition for the existence of $A^D$. If $A$ is a complex matrix, then $A^G$ is exactly the Moore-Penrose inverse of $A$}.   Also, our discussion is valid for all dual complex matrices and even dual quaternion matrices.
	
	In the next section, we review the knowledge of dual numbers and dual matrices, and the properties of DMPGI and MPDGI.    In Section 3, we study singular values of dual matrices.  Then, in Section 4, we study GMPI and its properties.  Finally, in Section 5, we discuss the relations and properties of GMPI,  DMPGI, and MPDGI.

	\section{Dual Numbers, Dual Number Vectors and Dual Number Matrices}
	
	The field of real numbers, the field of complex numbers, the set of quaternions, the set of dual real numbers and the set of dual complex numbers, the set of dual quaternions are denoted  by $\mathbb R$, $\mathbb C$, $\mathbb Q$, $\mathbb D$, $\mathbb {DC}$ and $\mathbb {DQ}$, respectively.    In this section, we cover quaternions and dual quaternions.   Readers who are not interested in them can neglect them.  This will not affect the reading.
	
	In the literature, there are at least two different definitions of dual complex numbers.  One of them defines the dual complex number multiplication as noncommutative.   Here, we adopt the concept that dual complex numbers are special dual quaternions, dual real numbers are special dual complex numbers, and the dual complex number multiplication is commutative.   For the difference between these two definitions, one may see \cite{QC23} for more discussions.

	A {\bf dual number} $a = a_s + a_d\epsilon$ has standard part $a_s$ and dual part $a_d$.   If both $a_s$ and $a_d$ are real numbers, then $a$ is a {\bf dual real number}.   If both $a_s$ and $a_d$ are complex numbers, then $a$ is a {\bf dual complex number}.   If both $a_s$ and $a_d$ are quaternions, then $a$ is a {\bf dual quaternion}. The symbol $\epsilon$ is the infinitesimal unit, satisfying $\epsilon^2 = 0$, and $\epsilon$ is commutative with real numbers.  If $a_s \not = 0$, then we say that $a$ is {\bf appreciable}.   Otherwise, we say that $a$ is {\bf infinitesimal}.
	
	For a dual complex number or a dual quaternion $a = a_s + a_d\epsilon$, its conjugate is defined as $a^* = a_s^* + a_d^*\epsilon$, i.e., both $a_s$ and $a_d$ take their conjugates.
	
	Suppose we have two dual numbers $a = a_s + a_d\epsilon$ and $b = b_s + b_d\epsilon$.   Then their sum is $a+b = (a_s+b_s) + (a_d+b_d)\epsilon$, and their product is $ab = a_sb_s + (a_sb_d+a_db_s)\epsilon$.
	The multiplication of dual real or complex numbers is commutative.   But the multiplication of dual quaternions is not commutative, this is due to the noncommutative property of quaternion multiplication.
	
	Suppose we have two dual real numbers $a = a_s + a_d\epsilon$ and $b = b_s + b_d\epsilon$.   By \cite{QLY22}, if
	$a_s > b_s$, or $a_s = b_s$ and $a_d > b_d$, then we say $a > b$.   Then this defines positive, nonnegative dual numbers, etc.
	In particular, for a dual number $a = a_s + a_d\epsilon$, its magnitude is defined as a nonnegative dual number
	$$|a| = \left\{ \begin{array}{ll} |a_s| + {\rm sgn}(a_s)a_d\epsilon, & \ {\rm if}\  a_s \not = 0, \\ |a_d|\epsilon, &   \ {\rm otherwise}.  \end{array}  \right.$$
	For any  dual   number $a=a_s+a_d\epsilon$ and dual number  $b=b_s+b_d\epsilon$ with $a_s\neq 0$, or $a_s=0$ and $b_s=0$, there is
	\begin{equation*}
		\frac{a_s+a_d\epsilon}{b_s+b_d\epsilon} =
		\left\{
		\begin{array}{ll}
			\frac{a_s}{b_s}+\left(  \frac{a_d}{b_s}- \frac{a_s}{b_s} \frac{b_d}{b_s}\right)\epsilon,   & \text{ if } b_s\neq 0, \\
			\frac{a_d}{b_d} +c\epsilon,  & \ \text{if } a_s= 0, b_s= 0,\\
		\end{array}
		\right.
	\end{equation*}
	where $c$ is an arbitrary   number.
	
	We   use $0$, ${\bf 0}$, and $O$ to denote a zero number, a zero vector, and a zero matrix, respectively.
	
	
	A dual number vector is denoted by $\vx = (x_1, \dots, x_n)^\top$.  If $x_1, \dots, x_n \in {\mathbb D}$, then we say that $\vx$ is a dual real vector.   If $x_1, \dots, x_n \in {\mathbb {DC}}$, then we say that $\vx$ is a dual complex vector.  If $x_1, \dots, x_n \in {\mathbb {DQ}}$, then we say that $\vx$ is a dual quaternion vector.  Its $2$-norm is defined as
	\begin{equation}\label{dual_norm}
		\|\vx\|_2 = \left\{\begin{array}{ll}
			\|\vx_s\|_2+{\frac{\vx_s^*\vx_d+\vx_d^*\vx_s}{2\|\vx_s\|_2}\epsilon,}
			& \ \mathrm{if} \  \vx_s \not = \0,\\
			\|\vx_d\|_2\epsilon, & \ \mathrm{if} \  \vx_s  = \0.
		\end{array}\right.
	\end{equation}
	Here $\vx^*$ is the transpose of $\vx$ if $\vx$ is a dual real vector, and the conjugate transpose of $\vx$ if $\vx$ is a dual complex vector or a dual quaternion vector.     We  may denote $\vx = \vx_s + \vx_d\epsilon$, where $\vx_s, \vx_d \in {\mathbb R}^n$, or ${\mathbb C}^n$, or ${\mathbb Q}^n$, respectively.
	If $\vx_s \not = \0$, then we say that $\vx$ is appreciable.
	The unit vectors in ${\mathbb R}^n$ are denoted as $\ve_1, \cdots, \ve_n$.   They are also unit vectors of ${\mathbb {C}}^n$, ${\mathbb {Q}}^n$ ${\mathbb {D}}^n$, ${\mathbb {DC}}^n$, and ${\mathbb {DQ}}^n$.	
	
	We say a dual number vector $\vx= (x_1, \cdots, x_n)^\top$ is a unit vector if $\|\vx\|_2=1$, or equivalently, $\|\vx_s\|_2=1$ and $\vx_s^*\vx_d=0$.
	
	Let $\vx, \vy \in {\mathbb D}^n$, or ${\mathbb {DC}}^n$, or ${\mathbb {DQ}}^n$.  If $\vx^* \vy = 0$, then we say that $\vx$ and $\vy$ are orthogonal.  Let $\vx^{(1)} = \vx_s^{(1)} + \vx_d^{(1)}\epsilon, \dots, \vx^{(k)} = \vx_s^{(k)} + \vx_d^{(k)}\epsilon \in {\mathbb D}^n$.  If $\vx_s^{(1)}, \dots, \vx_s^{(k)}$ are linearly independent, then we say that $\vx^{(1)}, \dots, \vx^{(k)}$ are appreciably linearly independent.

	A dual real matrix (or dual complex matrix, or dual quaternion matrix, respectively) is denoted by $A=A_s+A_d\epsilon\in {\mathbb {D}}^{m\times n}$ (or ${\mathbb {DC}}^{m\times n}$, or ${\mathbb {DQ}}^{m\times n}$, respectively).   Suppose that $m = n$.
	If $A_s$ is {invertible},
	then $A$ is also {invertible} and $A^{-1}=A_s^{-1}-A_s^{-1}A_dA_s^{-1}\epsilon$ {\cite{PS07, QC23}.}
	{The $F$-norm of $A\in\mathbb{D}^{m\times n}$ (or ${\mathbb {DC}}^{m\times n}$, or ${\mathbb {DQ}}^{m\times n}$, respectively) is
		\begin{equation}\label{dual_Fnorm}
			\|A\|_F = \left\{\begin{array}{ll}
				\|A_s\|_F+ \frac{tr(A_s^*A_d+A_d^*A_s)}{2\|A_s\|_F}\epsilon,
				& \ \mathrm{if} \  A_s \not = \0,\\
				\|A_d\|_F\epsilon, & \ \mathrm{if} \  A_s  = \0.
			\end{array}\right.
	\end{equation}}
	Let $A=A_s+A_d\epsilon\in {\mathbb {D}}^{m\times n}$.   Its transpose is defined as $A^\top = A_s^\top +A_d^\top\epsilon$.  Suppose that $m = n$. If $A = A^\top$, then $A$ is called a symmetric dual real matrix.  If $A^{-1} = A^\top$, then $A$ is called an orthogonal dual real matrix.  {
		We may show that
		\begin{equation*}
			\|UA\|_F=\|A\|_F
		\end{equation*}
		for any orthogonal dual {matrix} $U$ {and any dual real matrix $A$ with adequate dimension}.
		In fact, if $A_s=O$, there is $\|UA\|_F=\|U_sA_d\|_F\epsilon=\|A_d\|_F\epsilon=\|A\|_F$. If $A_s\neq O$, there is $\|UA\|_F^2=\|A\|_F^2$.
	}

	
	Similarly, let $A=A_s+A_d\epsilon\in {\mathbb {DC}}^{m\times n}$, or ${\mathbb {DQ}}^{m\times n}$.   Its conjugate transpose is defined as $A^* = A_s^* +A_d^*\epsilon$.  Suppose that $m = n$. If $A = A^*$, then $A$ is called a Hermitan  dual complex (or quaternion) matrix.  If $A^{-1} = A^*$, then $A$ is called a unitary dual complex ({or quaternion}) matrix.  {Again, for any unitary dual complex (quaternion, respectively) matrix $U$ and any dual complex (quaternion, respectively) matrix $A$ with adequate dimension, we have $\|UA\|_F=\|A\|_F$.}
	

	\subsection{Dual Generalized Inverses}
	
	For dual real matrix $A = A_s + A_d\epsilon \in {\mathbb D}^{m \times n}$, a dual real matrix $X$ is called an $\{i, \dots, j \}$- dual generalized inverse of $A$ if $X$ satisfies equations (i), \dots, (j) among (\ref{pen1})-(\ref{pen4}).  In \cite{PS07} (also see \cite{PV09}), Pennestr\`i and Stefanelli introduced the Moore-Penrose dual generalized inverse (MPDGI, as called in \cite{Wa21}) of $A$ as
	\eqref{MPDGI}.
	It was shown in \cite{PV09} that when $m > n$ {and $A_s$ is of full column   rank}, $A^PA = I_n$.  {Furthermore,  $A^P$  is a \{1,2,4\}- dual generalized inverse of $A$ \cite{FPU18}.}  This  can be verified  directly by (\ref{MPDGI}) and $A_s^+A_s = I_n$.
	Similarly, if $m < n$  and $A_s$ is of full row rank, $AA^P = I_m$ and $A^P$  is a \{1,2,3\}- dual generalized inverse of $A$ \cite{FPU18}.
	If $m=n$ and $A_s$ is nonsingular, then $A^P$ is also the DMPGI of $A$.

	De Falco et al. \cite{FPU18} showed that  $A^P$ in \eqref{MPDGI} is a \{2\}-dual generalized inverse of $A$, and gave  necessary and sufficient conditions for $A^P$ in \eqref{MPDGI} to be a \{1\}-, \{3\}-, \{4\}- dual generalized inverse of $A$, respectively.  Specifically,  $A^P$ in \eqref{MPDGI} is a \{1\}- dual generalized inverse of $A$ if and only if
	\begin{equation}\label{DMPGI:existcondition}
		(I_m-A_sA_s^+)A_d(I_n-A_s^+A_s)=O.
	\end{equation}
	
	Wang \cite{Wa21} presented   necessary and sufficient conditions for the dual matrix $A$ to have   DMPGI  and   necessary and sufficient conditions for the MPDGI to be a DMPGI of $A$.

	\begin{Thm}[Theorem 2.2, \cite{Wa21}]\label{thm:A+exists} Let $A=A_s+A_d\epsilon$. Then the following conditions are equivalent.
		\begin{itemize}
			\item[(a)]  The DMPGI  $A^+$ exists;
			\item[(b)] Equation \eqref{DMPGI:existcondition} holds;
			\item[(c)] $\mathrm{rank}\left[
			\begin{array}{cc}
				A_d & A_s \\
				A_s  & O
			\end{array}
			\right]=2\mathrm{rank}(A_s)$.
		\end{itemize}
		Furthermore, when $A^D$ exists, it holds that
		\begin{equation}\label{equ:DMPGI}
			A^D=A_s^+- A_s^+A_dA_s^+\epsilon +(A_s^\top A_s)^+A_d^\top(I_m-A_sA_s^+)\epsilon +(I_n-A_s^+A_s)A_d^\top(A_sA_s^\top)^+  \epsilon.
		\end{equation}
	\end{Thm}

	\begin{Thm}[Theorem 2.3, \cite{Wa21}]\label{thm:AP=A+}  Let $A=A_s+A_d\epsilon$. Then the following conditions are equivalent.
		\begin{itemize}
			\item[(a)]  The DMPGI  $A^D$ exists  and $A^D=A^P$;
			\item[(b)]  $(I_m-A_sA_s^+)A_d=O$ and $A_d(I_n-A_s^+A_s)=O$;
			\item[(c)] $\mathrm{rank}\left[
			\begin{array}{cc}
				A_s & A_d
			\end{array}
			\right] = \mathrm{rank}\left[
			\begin{array}{cc}
				A_s^\top & A_d^\top
			\end{array}
			\right]=\mathrm{rank}(A_s)$;
			
			\item[(d)] $\mathcal{R}(A_d)\subset \mathcal{R}(A_s)$ and $\mathcal{R}(A_d^\top)\subset \mathcal{R}(A_s^\top)$.
		\end{itemize}
	\end{Thm}
	
	\bigskip

	\section{Singular Value Decomposition of Dual Matrices}

	The following singular value decomposition theorem of dual quaternion matrices was given in \cite{QL23}.
	
	\begin{Thm}\label{DQSVD}
		Suppose that $A \in {\mathbb {DQ}}^{m \times n}$. Then there exists a unitary  dual quaternion matrix $V \in {\mathbb {DQ}}^{m \times m}$ and a unitary dual quaternion matrix $U \in {\mathbb {DQ}}^{n \times n}$ such that
		\begin{equation}
			\Sigma = V^* AU = \begin{bmatrix} \Sigma_t & O  \\ O & O \end{bmatrix},
		\end{equation}
		where $\Sigma_t\in {\mathbb{D}}^{t\times t}$ is a diagonal dual real matrix, with the form $$\Sigma_t ={\rm diag}\left(\mu_1, \cdots, \mu_r, \cdots, \mu_t \right),$$
		$r \le t \le \min \{ m , n \}$, $\mu_1 \ge \mu_2 \ge \cdots \ge \mu_r$ are positive appreciable dual numbers, and $\mu_{r+1} \ge \cdots \ge \mu_t$ are positive infinitesimal  dual numbers.    Counting possible multiplicities of the diagonal entries, the form $\Sigma_t$ is unique.
		
	\end{Thm}
	
	If we follow the derivation of \cite{QL23}, then we may get the dual complex and dual linear versions of the above theorem as follows.

	\begin{Thm}\label{DCSVD}
		Suppose that $A \in {\mathbb {DC}}^{m \times n}$. Then there exists a unitary  dual complex matrix $V \in {\mathbb {DC}}^{m \times m}$ and a unitary dual complex matrix $U \in {\mathbb {DC}}^{n \times n}$ such that
		\begin{equation}\label{equ:Sigma}
			\Sigma = V^* AU = \begin{bmatrix} \Sigma_t & O  \\ O & O \end{bmatrix},
		\end{equation}
		where $\Sigma_t\in {\mathbb{D}}^{t\times t}$ is a diagonal dual real matrix, with the form $$\Sigma_t ={\rm diag}\left(\mu_1, \cdots, \mu_r, \cdots, \mu_t \right),$$
		$r \le t \le \min \{ m , n \}$, $\mu_1 \ge \mu_2 \ge \cdots \ge \mu_r$ are positive appreciable dual numbers, and $\mu_{r+1} \ge \cdots \ge \mu_t$ are positive infinitesimal  dual numbers.    Counting possible multiplicities of the diagonal entries, the form $\Sigma_t$ is unique.
		
	\end{Thm}	
	
	\begin{Thm}\label{DSVD}
		Suppose that $A \in {\mathbb {D}}^{m \times n}$. Then there exists an orthogonal  dual real matrix $V \in {\mathbb {D}}^{m \times m}$ and an orthogonal dual real matrix $U \in {\mathbb {D}}^{n \times n}$ such that
		\begin{equation}
			\Sigma = V^\top AU = \begin{bmatrix} \Sigma_t & O  \\ O & O \end{bmatrix},
		\end{equation}
		where $\Sigma_t\in {\mathbb{D}}^{t\times t}$ is a diagonal dual real matrix, with the form $$\Sigma_t ={\rm diag}\left(\mu_1, \cdots, \mu_r, \cdots, \mu_t \right),$$
		$r \le t \le \min \{ m , n \}$, $\mu_1 \ge \mu_2 \ge \cdots \ge \mu_r$ are positive appreciable dual numbers, and $\mu_{r+1} \ge \cdots \ge \mu_t$ are positive infinitesimal  dual numbers.    Counting possible multiplicities of the diagonal entries, the form $\Sigma_t$ is unique.
		
	\end{Thm}
	
	Note that Theorems \ref{DCSVD} and \ref{DSVD} are not corollaries of Theorem \ref{DQSVD}.    We {cannot} obtain the conclusions on $U$ and $V$ are unitary dual complex matrices or orthogonal dual real matrices directly from Theorem \ref{DQSVD}.  But we can obtain Theorems \ref{DCSVD} and \ref{DSVD} by following the derivation steps of Theorem \ref{DQSVD} in \cite{QL23}.   Also, it is worth noting that in these three theorems, $\Sigma_t$ is always a diagonal dual real matrix.

	We call $\mu_1, \cdots, \mu_t$ and possibly $\mu_{t+1} = \cdots = \mu_{\min \{m, n \} }=0$, if $t < \min \{m, n\}$, the singular values of $A$, $t$ the rank of $A$, and $r$ the appreciable rank of $A$.
	In particular, we call $\mu_{r+1}, \dots, \mu_r$ the {\bf infinitesimal singular values} of $B$.  If
	$r=t$, then we say that $A$ is {\bf essential}.  In general, let
	\begin{equation}\label{equ:Ae}
		A_e = V\begin{bmatrix} \Sigma_r & O  \\ O & O \end{bmatrix}U^*.
	\end{equation}
	We call $A_e$ the {\bf essential part} of $A$, and
	{$$A_n = A-A_e= U_s\begin{bmatrix} O & O  \\  O &\Sigma_{2d} \end{bmatrix}V_s^* \epsilon,$$
		the {\bf nonessential part} of $A$,	where $\Sigma_{2d}={\rm diag}(\mu_{r+1,d},\cdots,\mu_{t,d},0,\cdots,0)$.}  By Theorems \ref{DQSVD}, \ref{DCSVD} and \ref{DSVD}, $A_e$ and $A_n$   are well-defined.   By \cite{QL23}, $A_e$ is the best rank-$r$ approximation of $A$.
	Furthermore, the nonessential part is always infinitesimal.
	Suppose $m = n = 1$.
	If the number $a$ is appreciable, then  its essential part is exactly $a$, and the nonessential part is zero; Otherwise,   the essential part is zero, and the nonessential part is exactly the dual part.

	\bigskip

	\section{A Genuine Extension of the Moore-Penrose Inverse}
	
	We now have the dual number version of the Moore-Penrose conditions (\ref{pen2}-\ref{pen4}) and (\ref{pene}).

	\begin{Def}
		Let $A=A_s+A_d\epsilon\in\mathbb {DC}^{m\times n}$ and $X=X_s+X_d\epsilon\in\mathbb {DC}^{n\times m}$. If  $X$  satisfies (\ref{pen2}-\ref{pen3}), i.e., (1e) $AXA = A_e$, (2) $XAX= X$, (3) $(AX)^*= AX$, (4) $(XA)^* = XA$.  We call $X$
		the {\bf genuine extension of the Moore-Penrose inverse} (GMPI) of $A$, denoted by $A^G$.
	\end{Def}
	
	{Consider the case that $m=n=1$ and $a=a_s+a_d\epsilon$. If  $a_s\neq0$, then $x=a^{-1}=a_s-a_s^{-2}a_d\epsilon$. If $a_s=0$ and $a_d\neq 0$, then $x=0$. If $a_s=a_d=0$, then $x=0$.
		
		If $A=\begin{bmatrix} \Sigma_t & O  \\ O & O \end{bmatrix}$ as in \eqref{equ:Sigma},
		where $\Sigma_t\in {\mathbb{D}}^{t\times t}$ is a diagonal dual real matrix, with the form $\Sigma_t ={\rm diag}\left(\mu_1, \cdots, \mu_r, \cdots, \mu_t \right),$
		$r \le t \le \min \{ m , n \}$, $\mu_1 \ge \mu_2 \ge \cdots \ge \mu_r$ are positive appreciable dual numbers, and $\mu_{r+1} \ge \cdots \ge \mu_t$ are positive infinitesimal  dual numbers.   Then we have
		$$A^G=A_e^D=\begin{bmatrix} \Sigma_r^{-1} & O  \\ O & O \end{bmatrix}=\begin{bmatrix} \Sigma_{r,s}^{-1} & O  \\ O & O \end{bmatrix}-\begin{bmatrix} \Sigma_{r,s}^{-2}\Sigma_{r,d} & O  \\ O & O \end{bmatrix}\epsilon,$$
		where $\Sigma_r ={\rm diag}(\mu_1, \cdots, \mu_r)$.
	}

	In the following theorem, we show that for any complex $A$, its GMPI exists uniquely.
	
	For a complex matrix $A_s$, let the SVD be $A_s=U_s\Sigma_sV_s^*$ with $\Sigma_s=\left[
	\begin{array}{cc}
		\Sigma_{1s} & O \\
		O  & O
	\end{array}
	\right]$ and $\Sigma_{1s}\in\mathbb R^{r\times r}$, then it is well known that its MP generalized inverse has the following form,
	\begin{equation*}
		A_s^+=V_s\left[
		\begin{array}{cc}
			\Sigma_{1s}^{-1} & O \\
			O  & O
		\end{array}
		\right]U_s^*.
	\end{equation*}
	Now we show that this result also holds true for GMPI of   dual complex  matrices.

	\begin{Thm}\label{thm:EssentialMPI}
		Let $A=A_s+A_d\epsilon\in\mathbb {DC}^{m\times n}$.  Suppose its SVD is $A = U\Sigma V^*$, where $U\in {\mathbb {DC}}^{m \times m}$ and $V \in {\mathbb {DC}}^{n \times n}$ are unitary dual matrices, and $\Sigma$ is a nonnegative diagonal dual matrix,
		{
			\begin{equation*}
				\Sigma =  \left[
				\begin{array}{cc}
					\Sigma_{1s}  &  \\
					& O
				\end{array}
				\right]+\left[
				\begin{array}{cc}
					\Sigma_{1d} &  \\
					& \Sigma_{2d}
				\end{array}
				\right]\epsilon, \Sigma_{1}=\Sigma_{1s}+\Sigma_{1d}\epsilon={\rm diag}\left(\mu_1, \cdots, \mu_r \right),
			\end{equation*}
			and $\Sigma_2=\Sigma_{2d}\epsilon={\rm diag}\left(\mu_{r+1,d}, \cdots, \mu_{t,d}, 0,\cdots, 0 \right)$,
		}
		$r \le t \le \min \{ m , n \}$, $\mu_1 \ge \mu_2 \ge \cdots \ge \mu_r$ are positive appreciable dual   numbers, and $\mu_{r+1} \ge \dots \ge \mu_t$ are positive infinitesimal  dual  real numbers.
		Let
		\begin{equation}\label{GMPI}
			A^G = V\Sigma^G U^*,
		\end{equation}
		and
		{$$\Sigma^G = \begin{bmatrix} \Sigma_1^{-1} & O  \\ O & O \end{bmatrix}= \begin{bmatrix} \Sigma_{1s}^{-1} & O  \\ O & O \end{bmatrix}-\begin{bmatrix} \Sigma_{1s}^{-2}\Sigma_{1d} & O  \\ O & O \end{bmatrix}\epsilon \in {\mathbb {DC}}^{m \times n}.$$}
		Here, we have $\mu^{-1}=\mu_{s}^{-1}-\mu_{s}^{-2}\mu_d\epsilon$.
		Then  $X = A^G$ is  the {\bf genuine extension of Moore-Penrose inverse} (GMPI) of $A$, and it is the unique solution of the Moore-Penrose conditions (1e) and (2-4), i.e., (\ref{pen2}-\ref{pen4}) and (\ref{pene}).
		
		Furthermore, if $A$ is complex, i.e., $A_d = O$, then $A^G = A_s^+$.
		
	\end{Thm}
	\begin{proof}
		We first show $X=V\Sigma^GU^*$ satisfies the conditions (1e), (2)-(4). Indeed,
		\begin{eqnarray*}
			AXA &=&U\Sigma V^*  V\Sigma^GU^* U\Sigma V^*  = U\Sigma \Sigma^G \Sigma V^*  =U\Sigma_e V^* = A_e,\\
			XAX &=&  V\Sigma^GU^* U\Sigma V^*  V\Sigma^GU^* = V\Sigma^G\Sigma\Sigma^G U^* = V\Sigma^G U^*=X,\\
			(AX)^* &=& (U\Sigma V^* V\Sigma^GU^*)^* = (UI_{r,m}U^*)^*=UI_{r,m}U^*=AX,\\
			(XA)^* &=& ( V\Sigma^GU^*U\Sigma V^*)^* = (VI_{r,n}V^*)^*=VI_{r,n}V^*=XA,
		\end{eqnarray*}
		where {$\Sigma_e={\rm diag}(\Sigma_1,O_{m-r,n-r})\in\mathbb{DC}^{m\times n}$, $I_{r,n}=\text{diag}(I_r,O_{n-r})\in\mathbb R^{n\times n}$.}
		Therefore, for any $A\in\mathbb{DC}^{m\times n}$, $A^G$ defined by \eqref{GMPI} is the GMPI of $A$.

		Furthermore, since $$X_s=V_s\left[
		\begin{array}{cc}
			\Sigma_{1s}^{-1} &   \\
			& O
		\end{array}
		\right]  U_s^* \text{ and }
		A-A_e= U_s\left[
		\begin{array}{cc}
			O &   \\
			& \Sigma_{2d}
		\end{array}
		\right]  V_s^*\epsilon, $$
		there is
		\begin{eqnarray*}
			XA-XA_e = X_s (A-A_e)_d\epsilon= O.
		\end{eqnarray*}
		
		The uniqueness of $X$ is proved as follows.
		If $X_1$ and $X_2$ both satisfy conditions (1e), (2)-(4). Then we have
		\begin{eqnarray*}
			X_1&=&X_1AX_1 = X_1A_eX_1=X_1AX_2 AX_1\\
			&=& X_1(AX_2)^*(AX_1)^*=  X_1(AX_1AX_2)^*\\
			&=&X_1(A_eX_2)^*=X_1(AX_2)^*\\
			&=&   X_1AX_2\\
			&=& X_1A_eX_2=X_1AX_2 AX_2 = (X_1A)^*(X_2A)^*X_2\\
			&=&(X_2 AX_1A)^*X_2 = (X_2 A_e)^*X_2=(X_2 A)^*X_2\\
			&=& X_2 AX_2 = X_2.
		\end{eqnarray*}

		By the definition (\ref{GMPI}), if $A$ is a complex matrix, i.e., $A_d = O$, then $A^G = A_s^+$.
		This completes the proof.
	\end{proof}
	
	Let $m=n=1$.  Then $A$ and $X$ are dual numbers.  If $A = a_d\epsilon$ is a nonzero infinitesimal dual number, where $a_d$ is a real or complex number, then we always have $AXA = 0$, i.e., (\ref{pen1}) cannot hold.  Thus, we  generalize (\ref{pen1}) to \eqref{pene}. The reason is given in the following theorem.
	
	\begin{Thm}
		Let $A=A_s+A_d\epsilon\in\mathbb {DC}^{m\times n}$.  Suppose its SVD is $A = U\Sigma V^*$, where $U\in {\mathbb {DC}}^{m \times m}$ and $V \in {\mathbb {DC}}^{n \times n}$ are unitary dual {complex} matrices, and $\Sigma$ is a nonnegative diagonal dual matrix, 	
		\begin{equation*}
			\Sigma =  \left[
			\begin{array}{ccc}
				\Sigma_{1s}  & &  \\
				& O & \\
				& & O
			\end{array}
			\right]+\left[
			\begin{array}{ccc}
				\Sigma_{1d} &  & \\
				& \Sigma_{2d} & \\
				& & O
			\end{array}
			\right]\epsilon,
		\end{equation*}
		$\Sigma_{1}=\Sigma_{1s}+\Sigma_{1d}\epsilon={\rm diag}\left(\mu_1, \cdots, \mu_r \right)$, and $\Sigma_2=\Sigma_{2d}\epsilon={\rm diag}\left(\mu_{r+1,d}, \cdots, \mu_{t,d}\right)$,
		
		$r \le t \le \min \{ m , n \}$, $\mu_1 \ge \mu_2 \ge \cdots \ge \mu_r$ are positive appreciable {dual numbers}, and $\mu_{r+1} \ge \dots \ge \mu_t$ are positive infinitesimal  dual numbers. Then we have
		\begin{equation}
			A^G\in \mathop{\arg\min}\limits_{X\in \mathbb {DC}^{n\times m}} \quad \|AXA-A\|_F.
		\end{equation}
	{Furthermore, $A^G$ is the minimum-norm solution.
		The optimal value is   $$\|AA^GA-A\|_F=\|\Sigma_{2d}\|_F\epsilon.$$
		If all singular values are appreciable, then we have  $AA^GA=A$.}
	\end{Thm}
	\begin{proof}
		Let $Y=V^* X U$, then we have
		\begin{eqnarray*}
			\|AXA-A\|_F &=& \|U\Sigma V^* X U\Sigma V^* - U\Sigma V^*\|_F\\
			&=& \|U\Sigma Y\Sigma V^* - U\Sigma V^*\|_F\\
			&=&   \| \Sigma Y \Sigma  -\Sigma\|_F\\
			&=&  \left\| \left[\begin{array}{cc}
				\Sigma_{1} Y_{11} \Sigma_{1}  -\Sigma_{1}   & \Sigma_{1} Y_{12} \Sigma_{2}    \\
				\Sigma_{2d} Y_{21} \Sigma_{1}  & \Sigma_{2d} Y_{22} \Sigma_{2}  -\Sigma_{2}
			\end{array} \right]\right\|_F\\
			&=&  \left\| \left[\begin{array}{cc}
				\Sigma_{1} Y_{11} \Sigma_{1}  -\Sigma_{1}   & \Sigma_{1} Y_{12} \Sigma_{2d} \epsilon    \\
				\Sigma_{2d} Y_{21} \Sigma_{1} \epsilon  & \Sigma_{2d} Y_{22} \Sigma_{2d} \epsilon^2  -\Sigma_{2d}  \epsilon
			\end{array} \right]\right\|_F.
		\end{eqnarray*}
		{Here, $Y_{11}\in\mathbb{DC}^{r\times r}$, $Y_{12}\in\mathbb{DC}^{r\times (t-r)}$,   $Y_{21}\in\mathbb{DC}^{(t-r)\times r}$, and $Y_{22}\in\mathbb{DC}^{(t-r)\times (t-r)}$.

			If $Y_{11s}= \Sigma_{1s}^{-1}$, then $\|AXA-A\|_F$ is an infinitesimal number.  By the total order in Qi, Ling and Yan \cite{QLY22}, an infinitesimal dual real number is smaller than any positive appreciable dual number. Hence,  we have
			$$ \Sigma_{1s}^{-1}=\mathop{\arg\min}\limits_{Y_{11s}\in \mathbb {C}^{r\times r}} \quad \| \Sigma Y \Sigma  -\Sigma\|_F$$
			and
			\begin{eqnarray*}
				\|AXA-A\|_F &=& \left\| \left[\begin{array}{cc}
					\Sigma_{1s} Y_{11d} \Sigma_{1s} +\Sigma_{1d}    & \Sigma_{1s} Y_{12d} \Sigma_{2d}    \\
					\Sigma_{2d} Y_{21d} \Sigma_{1s}    &   -\Sigma_{2d}
				\end{array} \right]\right\|_F \epsilon.
			\end{eqnarray*}
			Thus, there is
			\begin{eqnarray}
				-\Sigma_{1s}^{-2}\Sigma_{1d}, O,O= \mathop{\arg\min}\limits_{Y_{11d}, Y_{12d}, Y_{21d}} \quad \| \Sigma Y \Sigma  -\Sigma\|_F.
			\end{eqnarray}
			
			From the above discussion, we have the optimal value is $\|\Sigma_{2d}\|_F\epsilon$, and the optimal solution satisfies $Y_{11}=\Sigma_1^{-1}$, $Y_{12d}=O$, and $Y_{21d}=O$, i.e.,
			$$\left[
			\begin{array}{ccc}
				\Sigma_{1}^{-1}  & Y_{12s} & Y_{13} \\
				Y_{21s}& Y_{22} & Y_{23} \\
				Y_{31}& Y_{32} & Y_{33} \\
			\end{array}
			\right] \in\mathop{\arg\min}\limits_{Y\in \mathbb {DC}^{n\times m}} \quad \| \Sigma Y \Sigma  -\Sigma\|_F.$$
			
			Among the solutions, there is $$Y=\left[\begin{array}{cc}
				\Sigma_{1}^{-1}  &    \\
				& O  \\
			\end{array}\right],$$
			which is the minimum-norm solution.
			Equivalently, we have
			$$A^G=X=V\left[
			\begin{array}{cc}
				\Sigma_{1}^{-1}  &  \\
				& O
			\end{array}
			\right] U^*.$$
		}
		This completes the proof.
	\end{proof}

	Note that the discussion  in this section  is also valid for dual quaternion matrices.   We do not go into details here.

	\bigskip
	
	\section{GMPI, DMPGI and MPDGI}
	
	Below, we present a new sufficient and necessary condition for  the existence of DMPGI {and a new formulation of DMPGI}.
	
	\begin{Thm}\label{thm:existence_SVD}
		Let $A=A_s+A_d\epsilon\in\mathbb {DC}^{m\times n}$.  Then DMPGI of $A$ exists if and only if all singular values of $A$ are  appreciable, i.e., $A$ is substantial.
		{
			Furthermore, suppose $A=U\Sigma V^*$ is the SVD of $A$, $\Sigma={\rm diag}(\Sigma_r, O_{m-r,n-r})$, $\Sigma_r={\rm diag}(\mu_1,\cdots,\mu_r)$,  and $\mu_1,\cdots,\mu_r$ are appreciable.  Then there is
			\begin{equation}\label{equ:AD}
				A^D=V\Sigma^D U^*,\ \Sigma^D=\left[
				\begin{array}{cc}
					\Sigma_r^{-1} & O \\
					O  & O
				\end{array}
				\right].
			\end{equation}
		}
	\end{Thm}
	\begin{proof}
		By Wang \cite{Wa21}, the DMPGI of $A$ exists if and only if $(I_m-A_sA_s^+)A_d(I_n-A_s^+A_s)=O.$ Namely,
		\begin{eqnarray*}
			O &=& (I_m-A_sA_s^+)A_d(I_n-A_s^+A_s) \\
			&=&  U_s\left[
			\begin{array}{cc}
				O & O \\
				O  & I
			\end{array}
			\right]U_s^\top \left(U_s\Sigma_sV_d^*+U_s\Sigma_dV_s^*+U_d\Sigma_sV_s^*\right)  V_s\left[
			\begin{array}{cc}
				O & O \\
				O  & I
			\end{array}
			\right]V_s^\top \\
			&=& U_s\left[
			\begin{array}{cc}
				O & O \\
				O  & I
			\end{array}
			\right](\Sigma_sV_d^* V_s + \Sigma_d + U_s^* U_d\Sigma_s )\left[
			\begin{array}{cc}
				O & O \\
				O  & I
			\end{array}
			\right]V_s^\top \\
			&=&  U_s\left[
			\begin{array}{cc}
				O & O \\
				O  & \Sigma_{2d}
			\end{array}
			\right]V_s^\top.
		\end{eqnarray*}
		Here, $\Sigma_d=\text{diag}(\Sigma_{1d},\Sigma_{2d})$, $\Sigma_{1d}\in\mathbb R^{r\times r}$, and $\Sigma_{2d}\in\mathbb R^{(m-r)\times (n-r)}$.
		Hence, there is $\Sigma_{2d}=O$, or equivalently, all singular values of $A$ are  appreciable.
		
		{Furthermore, let $X=V\Sigma^D U^*$, then there is
			\begin{eqnarray*}
				AXA &=&U\Sigma V^*  V\Sigma^DU^* U\Sigma V^*  = U\Sigma \Sigma^D \Sigma V^*  =U\Sigma  V^* = A,\\
				XAX &=&  V\Sigma^DU^* U\Sigma V^*  V\Sigma^DU^* = V\Sigma^D\Sigma\Sigma^D U^* = V\Sigma^D U^*=X,\\
				(AX)^* &=& (U\Sigma V^* V\Sigma^DU^*)^* = (UI_{r,m}U^*)^*=UI_{r,m}U^*=AX,\\
				(XA)^* &=& ( V\Sigma^DU^*U\Sigma V^*)^* = (VI_{r,n}V^*)^*=VI_{r,n}V^*=XA,
			\end{eqnarray*}
			where {$\Sigma_e={\rm diag}(\Sigma_1,O_{m-r,n-r})\in\mathbb{DC}^{m\times n}$, $I_{r,n}=\text{diag}(I_r,O_{n-r})\in\mathbb R^{n\times n}$.}
			Hence, $V\Sigma^D U^*$ is the DMPGI of $A$.
		}
		
		This completes the proof.
	\end{proof}
	
	{By Theorem~\ref{thm:existence_SVD}, there is
		\begin{equation}\label{equ:An}
			A_n= U_s\begin{bmatrix} O & O  \\  O &\Sigma_{2d} \end{bmatrix}V_s^* \epsilon=(I_m - A_sA_s^+)A_d(I_n-A_s^+A_s) \epsilon,
		\end{equation}
		where $\Sigma_{2d}={\rm diag}(\mu_{r+1,d},\cdots,\mu_{t,d},0,\cdots,0)$.}

	The relationship between GMPI and DMPGI is given as follows.

	\begin{Thm}\label{thm:GMPI}
		Let $A=A_s+A_d\epsilon\in\mathbb D^{m\times n}$. Then there is
		\begin{equation}\label{equ:Ae=Ar+}
			A^G= A_e^D.
		\end{equation}
		In other words, if   $A$ is substantial, then  DMPGI exists and is the same as GMPI. Otherwise, DMPGI does not exist  and GMPI is the  DMPGI of the substantial approximation $A_e$.
	\end{Thm}
	\begin{proof}
		Equation \eqref{equ:Ae=Ar+} follows directly from \eqref{GMPI} {and \eqref{equ:AD}}.
		If all singular values of $A$ are appreciable, then $A_e=A$ and  the GMPI of $A$ is equal to  the DMPGI of $A$.
	\end{proof}

	Furthermore,   the relationship between MPDGI and DMPGI is given as follows.

	{\begin{Thm}\label{Lem:Ap=Atp}
			Denote $A_1= (A_sA_s^+) A (A_s^+A_s)$. Then we have
			$$A^P=A_1^P=A_1^D=A_1^G.$$
		\end{Thm}
		\begin{proof}
			By definition, we have
			\begin{eqnarray*}
				(A_1)_s&=&(A_sA_s^+) A_s (A_s^+A_s)=A_s.
			\end{eqnarray*}
			Hence
			\begin{eqnarray*}
				A_1^P&=&A_s^+-A_s^+(A_1)_dA_s^+\epsilon\\
				&=& A_s^+-A_s^+ (A_sA_s^+) A_d (A_s^+A_s)A_s^+\epsilon\\
				&=& A_s^+-A_s^+A_dA_s^+\epsilon\\
				&=& A^P.
			\end{eqnarray*}
			
			By  $(I_m-A_sA_s^+)(A_1)_d=O$ and $(A_1)_d(I_n-A_s^+A_s)=O$ and Theorem~\ref{thm:AP=A+}, we have $A_1^P=A_1^D$.
			
			Since $(I_m-A_sA_s^+)(A_1)_d(I_n-A_s^+A_s)=O$, $A_1$ is substantial. By  Theorem~\ref{thm:GMPI}, we have $A_1^D=A_1^G$.
			
			This completes the proof.
			
		\end{proof}
	}

	Note that the DMPGI of $\epsilon$ does not exist, but its GMPI is $0$, the same as its MPDGI. On the other hand, we have the following example.

	\begin{example}\label{exm 8.1}-- A dual real matrix $A$, whose  DMPGI does not exist, and MPDGI and GMPI are   different.
		
		Let $A = A_s+A_d\epsilon$, where
		\begin{equation}
			A_s = \left(\begin{array}{cc}
				1 & 0 \\
				0 & 0
			\end{array} \right)
			\text{ and }
			A_d = \left(\begin{array}{cc}
				0 & 1 \\
				1 & 1
			\end{array} \right).
		\end{equation}
		{The SVD of  $A$ is
			\begin{equation*}
				A=U\Sigma V^\top,   U=\left[\begin{array}{cc}
					1 & -\epsilon \\
					\epsilon & 1
				\end{array} \right], \Sigma=\left[\begin{array}{cc}
					1   &  \\
					& \epsilon
				\end{array} \right], V=\left[\begin{array}{cc}
					1 & -\epsilon \\
					\epsilon & 1
				\end{array} \right].
			\end{equation*}
			Further, there are
			\[\Sigma_e=\left[\begin{array}{cc}
				1   &  \\
				& 0
			\end{array} \right], \ \Sigma^G=\left[\begin{array}{cc}
				1   &  \\
				& 0
			\end{array} \right].\]
		}
		
		The MPDGI is $A^P=A_s=A_s^+$. In other words, the dual part is completely ignored in MPDGI.
		On the other hand, there are
		\begin{equation*}
			A_e= {U\Sigma_e V^\top=}\left(\begin{array}{cc}
				1 & \epsilon \\
				\epsilon & 0
			\end{array} \right) \text{ and } A^G={U\Sigma^GV^\top=} \left(\begin{array}{cc}
				1 & \epsilon \\
				\epsilon & 0
			\end{array} \right).
		\end{equation*}
	\end{example}
	
	Both MPDGI and GMPI exist for all dual real matrices. However, they may be different, as  shown in Example~\ref{exm 8.1}. We now study their relationship in the following theorem.
	\begin{Thm}
		Let $A=A_s+A_d\epsilon\in\mathbb D^{m\times n}$. Then GMPI and MPDGI   are   equal if and only if
		\begin{equation}\label{equ:AP=Ar+}
			(A_sA_s^+)A_d(I_n-A_s^+A_s)=O \text{ and }  (I_m-A_sA_s^+)A_d(A_s^+A_s)=O.
		\end{equation}
		{
			Otherwise, if \eqref{equ:AP=Ar+} is not true, there is
			\begin{equation}
				A^P=A_{2}^G,
			\end{equation}
			where $A_{2}=A-(A_sA_s^+)A_d(I_n-A_s^+A_s)\epsilon-(I_m-A_sA_s^+)A_d(A_s^+A_s)\epsilon$.
		}
		\begin{proof}
			Consider the following two cases.
			
			Case (a),  $A$ is substantial. Then GMPI is equal to DMPGI.
			Then $(I_m-A_s^+A_s^+)A_d(I_n-A_s^+A_s)=O$. Combining it with \eqref{equ:AP=Ar+}, we have $(I_m-A_sA_s^+)A_d=O$ and $A_d(I_n-A_s^+A_s)=O$.
			Then by Wang \cite{Wa21},  MPDGI is equal to DMPGI. Hence,  we have  GMPI and MPDGI   are equal to each other.
			
			Case (b),  $A$ is not substantial. Then we have $A^G=A_e^D$. By \eqref{equ:AP=Ar+}, there is
			\begin{eqnarray*}
				A_e &=& A - (I_m-A_sA_s^+)A(I_n-A_s^+A_s)\\
				&=&  (A_sA_s^+)A(A_s^+A_s).
			\end{eqnarray*}
			Hence, it follows from Theorem~\ref{Lem:Ap=Atp} that
			\begin{equation*}
				A^P=A_e^D=A^G.
			\end{equation*}

			Furthermore, we have equation \eqref{equ:AP=Ar+} holds for $A_2$ and
			\begin{equation*}
				A^P=A_2^P=A_2^G,
			\end{equation*}
			where the first equality follows directly from the definitions of MPDGI and $A_2$, and the second equality follows from the first part of this theorem.
			
			This completes the proof.
		\end{proof}
	\end{Thm}
	It should be noted that    condition  \eqref{equ:AP=Ar+} is loose than the condition for the equivalence of MPDGI and DMPGI in Theorem~\ref{thm:AP=A+}, i.e.,
	$$(I_m-A_s^+A_s^+)A_dA_s^+A_s=O, \text{ and } A_s^+A_s^+A_d(I_n-A_s^+A_s)=O,$$
	since \eqref{equ:AP=Ar+} does not require $(I_m-A_s^+A_s^+)A_d(I_n-A_s^+A_s)=O$.
	We summarize the relationships between MPDGI, DMPGI, and GMPI   as follows.
	
	{\bf Remark.} Let $A=A_s+A_d\epsilon\in\mathbb D^{m\times n}$ and
	\begin{eqnarray}
		\nonumber M_2 &=& (A_sA_s^+)A_d(I_n-A_s^+A_s),\\
		M_3 &=& (I_m-A_sA_s^+)A_d(A_s^+A_s), \label{equ:M1234}\\
		\nonumber M_4 &=& (I_m-A_sA_s^+)A_d(I_n-A_s^+A_s).
	\end{eqnarray}
	Then we have the following results.
	\begin{itemize}
		\item[(i)]   If $M_4=O$, then DMPGI exists and GMPI is equal to DMPGI. Otherwise, $A^G=(A-M_4\epsilon)^D$.
		\item[(ii)] GMPI is equal to MPDGI  if and only if $M_2=0$ and $M_3=O$. Otherwise, $A^P=(A-M_2\epsilon-M_3\epsilon)^G$.
		\item[(iii)] GMPI is equal to MPDGI and DMPGI if and only if $M_2=O$, $M_3=O$, and $M_4=O$. Otherwise, $A^P=A_1^P=A_1^D=A_1^G$, where $A_1=(A-M_2\epsilon-M_3\epsilon-M_4\epsilon)$.
	\end{itemize}

	\begin{example}
		Let $A = A_s+A_d\epsilon$, where
		\begin{equation}
			A_s = \left[\begin{array}{cc}
				1 & 0 \\
				0 & 0
			\end{array} \right]
			\text{ and }
			A_d = \left[\begin{array}{cc}
				1 & 1 \\
				1 & 1
			\end{array} \right].
		\end{equation}
		For this example,
		\begin{equation*}
			M_2=  \left[\begin{array}{cc}
				0 & 1 \\
				0 & 0
			\end{array} \right],
			M_3=  \left[\begin{array}{cc}
				0 & 0 \\
				1 & 0
			\end{array} \right],
			M_4=  \left[\begin{array}{cc}
				0 & 0 \\
				0 & 1
			\end{array} \right].
		\end{equation*}
		The SVD of  $A$ is
		\begin{equation*}
			A=U\Sigma V^\top,   U=\left[\begin{array}{cc}
				1 & -\epsilon \\
				\epsilon & 1
			\end{array} \right], \Sigma=\left[\begin{array}{cc}
				1+\epsilon  &  \\
				& \epsilon
			\end{array} \right], V=\left[\begin{array}{cc}
				1 & -\epsilon \\
				\epsilon & 1
			\end{array} \right].
		\end{equation*}
		Further,  there are
		\[\Sigma_e=\left[\begin{array}{cc}
			1+\epsilon   &  \\
			& 0
		\end{array} \right], \ \Sigma^G=\left[\begin{array}{cc}
			1-\epsilon   &  \\
			& 0
		\end{array} \right].\]
		Hence, the DMPGI does not exist,  the GMPI satisfies
		\begin{equation*}
			A^G=(A-M_4\epsilon)^D = U\Sigma^G V^\top = \left[\begin{array}{cc}
				1-\epsilon  & \epsilon \\
				\epsilon & 0
			\end{array} \right],
		\end{equation*}
		and the MPDGI satisfies
		\begin{equation*}
			A^P= (A-M_2\epsilon-M_3\epsilon-M_4\epsilon)^P= \left[\begin{array}{cc}
				1-\epsilon  & 0 \\
				0 & 0
			\end{array} \right].
		\end{equation*}
		In other words, the GMPI may keep more information about the original matrix.
	\end{example}
	

	At last, we show  the relationships between MPDGI, DMPGI, and GMPI in Figure~\ref{fig:DGI_3cases}.
	\begin{figure}
		\centering
		\includegraphics[width=0.9\linewidth]{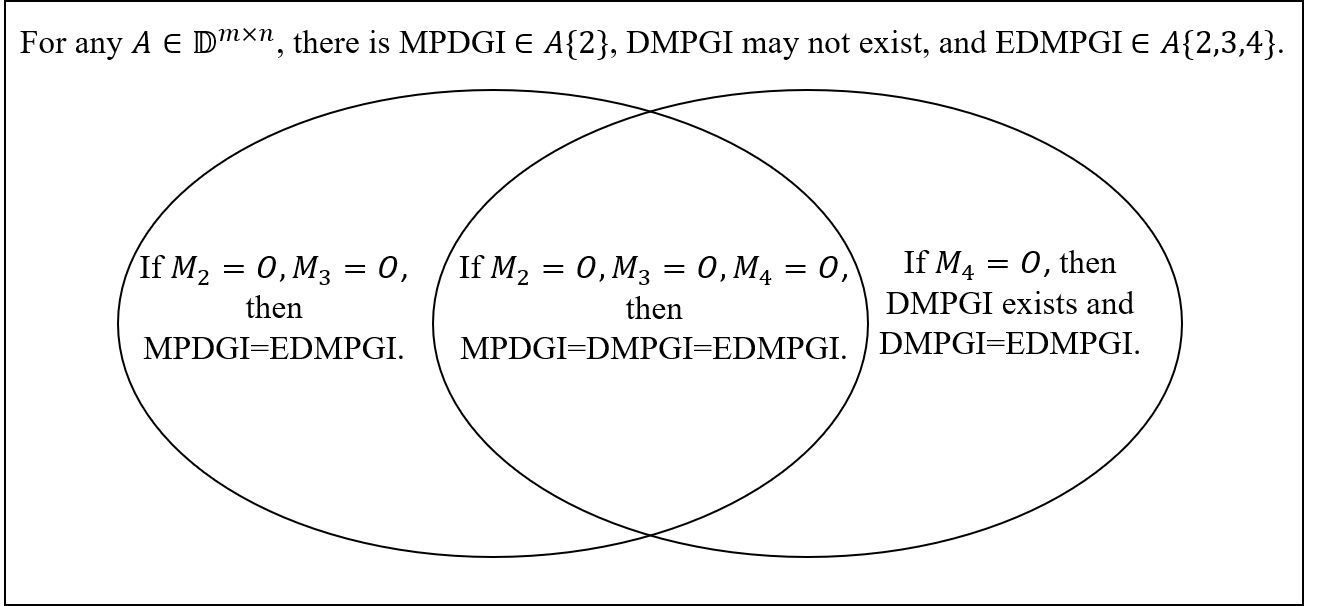}
		\caption{Relationships between MPDGI, DMPGI, and GMPI. Here, $M_2,\dots,M_4$ are defined by \eqref{equ:M1234}.}
		\label{fig:DGI_3cases}
	\end{figure}

	\bigskip
	
	{
		

	}

	

	\bigskip
	


\end{document}